\documentclass[12pt]{amsart}
\usepackage{amssymb}

\swapnumbers

\newtheorem{thm}{Theorem}[section]
\newtheorem{lem}[thm]{Lemma}

\newtheorem{prop}[thm]{Proposition}
\newtheorem{example}[thm]{Example}

\newtheorem*{prob*}{Open problem}

\theoremstyle{definition}

\newtheorem{defi}[thm]{Definition}

\theoremstyle{remark}

\newtheorem{rem}[thm]{Remark}
\newtheorem*{rem*}{Remark}

\DeclareMathOperator{\s}{span}

\DeclareMathOperator{\Hom}{Hom}

\newcommand{\kringel}{\mathbin{\raise1pt\hbox{$\scriptstyle\circ$}}} 
\newcommand{\pkt}{\mathbin{\raise0pt\hbox{$\scriptstyle\bullet$}}}

\newcommand{\C}{\mathbb{C}}

\newcommand{\N}{\mathbb{N}}

\newcommand{\ad}{\mathop{\rm ad}}

\newcommand{\La}{\mathfrak{a}}

\newcommand{\Lf}{\mathfrak{f}}
\newcommand{\Lg}{\mathfrak{g}}
\newcommand{\Lh}{\mathfrak{h}}

\newcommand{\Ln}{\mathfrak{n}}

\newcommand{\Lz}{\mathfrak{z}}

\newcommand{\LA}{\mathfrak{A}}
\newcommand{\LF}{\mathfrak{F}}

\newcommand{\CA}{\mathcal{A}}

\newcommand{\CL}{\mathcal{L}}

\newcommand{\CF}{\mathcal{F}}
\newcommand{\CI}{\mathcal{I}}

\newcommand{\mi}{\boldsymbol{-}}  
\newcommand{\im}{\mathop{\rm im}}

\newcommand{\al}{\alpha}
\newcommand{\be}{\beta}
\newcommand{\ga}{\gamma}
\newcommand{\de}{\delta}

\newcommand{\la}{\lambda}
\newcommand{\om}{\omega}

\newcommand{\ov}{\overline}

\newcommand{\ra}{\rightarrow}  

\renewcommand{\phi}{\varphi}

\begin{document}


\title[Affine cohomology classes]{Affine cohomology classes for filiform Lie algebras} 

\author[D. Burde]{Dietrich Burde}
\address{Mathematisches Institut\\
  Heinrich-Heine-Universit\"at\\
  Universit\"atsstr. 1\\
  40225 D\"us\-sel\-dorf\\
  Germany}
\email{dietrich@math.uni-duesseldorf.de}


\begin{abstract}

We classify the cohomology spaces $H^2(\Lg,K)$ for all
filiform nilpotent Lie algebras of dimension $n\le 11$ over $K$
and for certain classes of algebras of dimension $n\ge 12$.
The result is applied to the determination of
affine cohomology classes $[\om]\in H^2(\Lg,K)$.
We prove the general result that the
existence of an affine cohomology class implies an affine
structure of canonical type on $\Lg$, 
hence a canonical left-invariant affine structure
on the corresponding nilpotent Lie group.
For certain filiform algebras the absence of an affine
cohomology class implies the nonexistence of any affine
structure. Of particular interest are algebras $\Lg$
with minimal Betti numbers $b_1(\Lg)=b_2(\Lg)=2$.
\end{abstract}

\maketitle

\section{Introduction}

Left-invariant affine structures on Lie groups play an important role
in the study of fundamental groups of compact affine manifolds
and in the study of affine crystallographic groups. 
Milnor asked in his paper \cite{M2} on fundamental groups of complete 
affine manifolds whether or not every solvable Lie group admits a complete
left-invariant affine structure. There was evidence that the answer should 
be positive. Auslander \cite{A2} had proved a converse statement:
A Lie group admitting a complete left-invariant affine structure is solvable.
Milnor's question became known as a conjecture.
After a long history it was finally answered negatively.
It was proved that there exist nilpotent Lie groups without any left-invariant
affine structures, see \cite{B2},\cite{BG},\cite{BU3}.
All the known counterexamples are filiform nilpotent Lie groups
of dimension $10\le n \le 13$.
The result also implies that there exist
finitely generated torsionfree nilpotent groups which are not the fundamental
group of any compact complete affine manifold.
Moreover it follows that there exist nilmanifolds without any affine or
projective structure. 
Besides the counterexamples there do not exist many results on
Milnor's question. It is still unknown if there exist counterexamples
in every dimension $n\ge 10$.  
The problem can be formulated in purely algebraic terms: 

\begin{defi}
Let $A$ be a vector space
over a field $K$ and let $A \times A \rightarrow A,
(x,y) \mapsto x\cdot y$ be a $K$--bilinear product which satisfies
\begin{equation}\label{lsa1}
x\cdot (y\cdot z)-(x\cdot y)\cdot z= y\cdot (x\cdot z)-(y\cdot x)\cdot z
\end{equation}
for all $x,y,z \in A$. Then $A$ together with the product
is called {\it left-symmetric algebra} or
LSA. The product is also called left-symmetric.
\end{defi}
The term left-symmetric becomes evident, if we rewrite condition
$(\ref{lsa1})$ as
$(x,y,z)=(y,x,z)$, where $(x,y,z)=x\cdot (y\cdot z)-(x\cdot y)\cdot z$.

\begin{defi}\label{affine}
An {\it affine structure} or {LSA--structure} on a Lie algebra
$\Lg$ is a $K$--bilinear product $\Lg \times \Lg \rightarrow \Lg$
which is left-symmetric and satisfies
\begin{equation}\label{lsa2}
[x,y]=x\cdot y -y\cdot x
\end{equation}
\end{defi}

\begin{prop}
Let $G$ be a simply connected Lie group with Lie algebra $\Lg$.
There is a canonical one-to-one correspondence between left-invariant 
affine structures on $G$ and affine structures on $\Lg$. 
\end{prop}

The algebraic version of Milnor's question is the following:
which Lie algebras over a field $K$ do admit affine structures. 
As mentioned above it is not quite true that all solvable
Lie algebras admit an affine structure, although they tend to
admit one. The case of reductive Lie algebras also is very interesting.
There is a large literature on the subject, see for example \cite{BU2} 
and the references cited within. 
Milnor's question is related to a refinement of Ado's theorem
for finite-dimensional Lie algebras. If a Lie algebra $\Lg$ of dimension
$n$ admits an affine structure then $\Lg$ possesses a {\it faithful}
Lie algebra module of dimension $n+1$. The counterexamples to the
existence of affine structures rely on the fact that not all nilpotent
Lie algebras possess such a faithful module.
It is however still very difficult to determine which nilpotent Lie
algebras admit such modules. It requires a lot of computations to
find out that just one particular nilpotent Lie algebra does not have
a faithful module of small dimension. In general the question seems
to be completely hopeless unless the Lie algebra $\Lg$ is $2$-step, $3$-step
or $n-1$-step nilpotent, where $n$ denotes the dimension of $\Lg$.
However, in the first two cases there always exist an affine structure and hence
a faithful module of dimension $n+1$. The last case corresponds to filiform 
algebras. It is well known \cite{AS} that filiform algebras play also an 
interesting role in the study of Betti numbers $b_i(\Lg)$ of nilpotent 
Lie algebras $\Lg$. They produce lower bounds for the Betti
numbers. More precisely computations have shown that for any small $n$, there exists 
a filiform algebra $\Lf_n$ such that $b_i(\Lf_n)\le b_i(\Lg)$ for all $i$ and all 
nilpotent Lie algebras $\Lg$ of dimension $n$. 
Such filiform Lie algebras very often do not admit any affine structure.
The first step is to study the cohomology spaces $H^2(\Lg,K)$. 
We will compute the cohomology for all filiform Lie algebras of dimension
$n\le 11$ and for filiform Lie algebras of dimension $n\ge 12$ satisfying
certain properties. Here the algebras $\Lg$ of a certain subclass 
have minimal Betti numbers $b_1(\Lg)$ and $b_2(\Lg)$ equal to $2$. 
We conjecture that these algebras for $n\ge 13$ do not admit any affine structure.
We have proved it for $n=13$ so far.\\
The present article contains some results without proofs, which will appear
in an extended version.

\section{Filiform nilpotent Lie algebras}

In the study of nilpotent Lie algebras the filiform algebras
play an important role.

\begin{defi}
Let $\Ln$ be a Lie algebra over a field $K$. The lower central series
$\{\Ln^k\}$ of $\Ln$ is defined by
$\Ln^0=\Ln ,\; \Ln^k =[\Ln^{k-1},\Ln], \; k\ge 1.$
The integer $p$ is called {\it nilindex} of $\Ln$ if $\Ln^{p}=0$ and
$\Ln^{p-1}\ne 0$. In that case $\Ln$ is called $p$--step nilpotent.
A nilpotent Lie algebra $\Ln$ of dimension $n$ and nilindex $p=n-1$ is
called {\it filiform}.
\end{defi}

\begin{rem}
If we denote the {\it type} of a nilpotent Lie algebra
$\Ln$ by $\{p_1,p_2,\ldots ,p_r\}$ where $\dim (\Ln^{i-1}/\Ln^i)=
p_i$ for all $i=1,\ldots ,r$, then the filiform Lie algebras
are just the algebras of type $\{2,1,1,\ldots,1\}$. That explains the
name "filiform" which means threadlike. Note that the center of a
filiform Lie algebra is one-dimensional.
\end{rem}

\begin{example}\label{std}
Let $L=L(n)$ be the $n$-dimensional Lie algebra defined by
\begin{align*}
[e_1,e_i] & = e_{i+1}, \; i=2,\dots , n-1
\end{align*}
where $(e_1,\dots , e_n)$ is a basis of $L(n)$ and the undefined
brackets are zero. This is called the {\it standard graded filiform}.
\end{example}

It turns out to be useful to consider {\it adapted bases} for filiform
Lie algebras. We need a few definitions:

\begin{defi}
Let $\CI_n$ be an index set given by
\begin{align*}
\CI_n^0 &=\{(k,s)\in \N \times \N \mid 2 \le k \le [n/2],\,
2k+1 \le s \le n \},\\
\CI_n& =\begin{cases}
\CI_n^0 & \text{if $n$ is odd},\\
\CI_n^0 \cup \{(\frac{n}{2},n)\} & \text{if $n$ is even}.
\end{cases}
\end{align*}
\end{defi}

Let $(e_1,\ldots , e_n)$ be a basis of $L(n)$ with $[e_1,e_{i}]=e_{i+1}$
for $2\le i\le n-1$.
For any element $(k,s)\in \CI_n$ we can associate a $2$--cocycle
$\psi_{k,s}\in Z^2(L,L)$ for the Lie algebra cohomology with
coefficients in the adjoint module $L$ as follows:
\begin{align*}
\psi_{k,s}(e_1\wedge e_i) & = 0,\\
\psi_{k,s}(e_k\wedge e_{k+1}) & = e_s
\end{align*}
for $1\le i\le n,\,  2\le k\le n-1$.
Then the condition $\psi_{k,s}\in Z^2(L,L)$ for basis vectors
$e_1,e_i,e_j$ with $2\le i,j$ is given by
\begin{equation*}
[e_1, \psi_{k,s}(e_i\wedge e_j)] = \psi_{k,s}([e_1,e_{i}]\wedge e_j)+
\psi_{k,s}(e_i \wedge [e_1,e_{j}])
\end{equation*}
and we obtain the following formula:
\begin{equation}\label{psiks}
\psi_{k,s}(e_i\wedge e_j) =\begin{cases}
\displaystyle{(-1)^{k-i} \binom{j-k-1}{k-i}(\ad e_1)^{(j-k-1)-(k-i)}e_s} &
\text{if $2\le i\le k<j\le n$ },\\[0.2cm]
\displaystyle{0} & \text{otherwise}.
\end{cases}
\end{equation}
The $\psi_{k,s}$ defined by \eqref{psiks} in fact lie
in $Z^2(L,L)$. The following result is due to Vergne \cite{V}, where
$K=\C$:
\begin{lem}\label{vergne}
Any $n$--dimensional filiform Lie algebra is isomorphic to a
Lie algebra $L_{\psi}$ with basis $(e_1,\ldots , e_n)$ whose Lie brackets
are given by
\begin{equation}\label{apt}
[e_i,e_j]=[e_i,e_j]_L + \psi (e_i\wedge e_j),\; 1\le i,j \le n.
\end{equation}
Here $\psi$ is a $2$--cocycle which can be expressed by
\begin{equation*}
\psi=\sum_{(k,s)\in \CI_n}\al_{k,s} \psi_{k,s}
\end{equation*}
with $\al_{k,s}\in K$. The $2$--cocycle $\psi\in Z^2(L,L)$ defines an
infinitesimal deformation $L_{\psi}$ of $L$.
\end{lem}

\begin{defi}
A basis $(e_1,\ldots ,e_n)$ of an $n$--dimensional filiform Lie algebra
is called {\it adapted}, if the brackets relative to this basis are given
by $(\ref{apt})$ with a $2$--cocycle
$\psi=\sum_{(k,s)\in \CI_n}\al_{k,s} \psi_{k,s}$.
\end{defi}

Using Lemma $\ref{vergne}$ and \eqref{psiks} we obtain:

\begin{lem}
All brackets of an $n$--dimensional filiform Lie algebra in an adapted
basis $(e_1,\ldots ,e_n)$ are determined by the brackets
\begin{align*}
[e_1,e_i] & =e_{i+1}, \quad i=2,\dots ,n-1 \\
[e_k,e_{k+1}] & = \al_{k,2k+1}e_{2k+1}+\ldots  + \al_{k,n}e_n,\quad
2\le k \le
[(n-1)/2] \\
[e_{\frac{n}{2}},e_{\frac{n+2}{2}}] & = \al_{\frac{n}{2},n}e_n, \quad
\text{ if }n \equiv 0 (2)
\end{align*}
\end{lem}

\begin{lem}
The brackets of an $n$--dimensional filiform Lie algebra in an adapted
basis are given by:
\begin{align}\label{lie}
[e_1,e_i] & =e_{i+1}, \quad i=2,\dots ,n-1 \\
[e_i,e_j] & =\sum_{r=1}^n\biggl(\;\sum_{\ell=0}^{[(j-i-1)/2]} (-1)^\ell
{ j-i-\ell-1 \choose \ell}\al_{i+\ell,\, r-j+i+2\ell+1}\biggr)e_r,
 \quad 2 \le i<j \le n.
\end{align}
where the constants $\al_{k,s}$ are zero for all pairs $(k,s)$ not in $\CI_n$.
\end{lem}

Note that an adapted basis for a filiform Lie algebra is not unique.
Nevertheless we can associate coefficients $\{\al_{k,s} \mid (k,s) \in \CI_n\}$
to a filiform Lie algebra with respect to an adapted basis.
We have $(n-3)^2/4$ parameters if $n$ is odd, and $(n^2-6n+12)/4$ if $n$ is even. 
The Jacobi identity defines certain equations with polynomials in $K[\al_{k,s}]$.
If $n<8$, there are no equations, i.e., the Jacobi identity is
satisfied automatically.
In general, with respect to an adapted basis, the polynomial equations are
much simpler than usual. As an example, for filiform Lie algebras of dimension
$9$, the Jacobi identity with respect to $\{\al_{k,s} \mid (k,s) \in \CI_9\}$ is given by
the single equation $\al_{4,9}(2\al_{2,5}+\al_{3,7})-3\al_{3,7}^2=0.$

\begin{defi}
Let $V$ be a vector space of dimension $n$ over an algebraically closed
field $K$ of characteristic zero, with fixed
basis $(e_1,\ldots , e_n)$.
A Lie algebra structure on $V$ determines a multiplication table relative
to the basis. If
\begin{equation*}
[e_i,e_j] = \sum_{k=1}^n c_{ij}^k e_k
\end{equation*}
then the point $(c_{ij}^k) \in K^{n^3}$ is called a {\it Lie algebra law}.
\end{defi}

The constants $c_{ij}^k$ are subject to algebraic equations given by
the skew-symmetry and the Jacobi identity of the Lie bracket.
They define a certain Zariski-closed set in $n^3$--dimensional affine
space with coordinates $c_{ij}^k,\; 1\le i,j,k \le n$.
The set of all Lie algebra laws is often called the {\it variety of
Lie algebra laws} and is denoted by $\CL_n(K)$.

\begin{defi}
Denote by $\CF_n(K)$ the Zariski-open subset of $\CL_n(K)$ defining 
$n$--di\-men\-sio\-nal 
filiform nilpotent Lie algebras over $K$. Let $\CA_n(K)$ denote the subset of
$\CF_n(K)$ consisting of elements which are the structure constants of a
filiform Lie algebra with respect to an adapted basis.
If $\la\in\CF_n(K)$ then we denote the corresponding Lie algebra by
$\Lg_{\la}$. Denote the class of $n$--dimensional filiform Lie algebras
over $K$ by $\LF_n(K)$.
\end{defi}

Lemma $\ref{vergne}$ implies:
\begin{lem}\label{pro}
Let $\Lg\in \LF_n(K)$. Then there exists a basis $(e_1,\ldots ,e_n)$ such
that the corresponding Lie algebra law belongs to $\CA_n(K)$.
\end{lem}

Investigating affine structures on filiform Lie algebras
it turns out that the following subclasses are of importance.
Let $\Lg$ be a filiform Lie algebra of dimension $n\ge 7$ and
$\Lg^1=[\Lg,\Lg],\, \Lg^k=[\Lg^{k-1},\Lg]$ for $k\ge 2$. The following
properties are isomorphism invariants of $\Lg$:
\begin{itemize}
\item[(a)] $\Lg$ contains a one-codimensional subspace $U \supseteq \Lg^1$
such that $[U,\Lg^1]\subseteq \Lg^4$.
\item[(b)] $\Lg$ contains {\it no} one-codimensional subspace $U \supseteq \Lg^1$
such that $[U,\Lg^1]\subseteq \Lg^4$.
\item[(c)] $\Lg^{\frac{n-4}{2}}$
is abelian, where $n$ is even.
\item[(d)] $[\Lg^1,\Lg^1]\subseteq \Lg^6$.
\end{itemize}     

These properties can more naturally be formulated in terms of structure 
constants of an adapted basis. 

\begin{defi}
Let $\LA_n^1(K)$ denote the class of filiform Lie algebras of dimension
$n\ge 12$ satisfying properties (b),(c),(d). 
Let $\LA_n^2(K)$ denote the class of filiform Lie algebras of dimension
$n\ge 12$ satisfying properties (b),(c), but not property (d).
\end{defi}

These two classes are disjoint, in the sense that a
Lie algebra from the first class cannot be isomorphic to one of the
second class.
The algebras of $\LA_n^1(K)$ and $\LA_n^2(K)$ have some remarkable
properties concerning central extensions and affine structures.

\section{Affine cohomology classes}

In this section we will prove that the existence of
affine cohomology classes in $H^2(\Lg,K)$ for filiform Lie algebras $\Lg$
implies the existence of a canonical affine structure on $\Lg$.
It is then very interesting to study filiform Lie algebras with
minimal second Betti number $b_2(\Lg)=2$.
Such algebras do not admit an affine cohomology
class and hence no affine structure of canonical type. However,
in order to ensure that there exists no other affine structures
one needs additional conditions.\\[0.2cm]
Let us quickly review Lie algebra cohomology, for details see \cite{KNA}.
Denote by $\Lg$ a Lie algebra over $K$.
Denote by $M$ an $\Lg$--module with action $\Lg \times M \rightarrow M$,
$(x,m)\mapsto x \pkt m$.
The space of $p$--cochains is defined by
\begin{equation*}
C^p(\Lg,M)=\begin{cases}
\Hom_K(\Lambda^p\Lg,M) & \text{if $p\ge0$},\\
0 & \text{if $p<0$}.
\end{cases}
\end{equation*}
The {\it coboundary operators} $d_p: C^p(\Lg,M) \rightarrow C^{p+1}(\Lg,M)$
are defined by
\begin{equation*}\label{xx}
\begin{split}
(d_p\om)(x_1\wedge \dots \wedge x_{p+1}) & = \sum_{1\le r<s\le p+1}
(-1)^{r+s}\om ([x_r,x_s]\wedge x_1\wedge \dots \wedge \widehat{x_r}\wedge
\dots \wedge \widehat{x_s}\wedge \dots \wedge x_{p+1})\\
& \quad + \; \sum_{t=1}^{p+1} (-1)^{t+1}x_t \pkt \om (x_1\wedge \dots \wedge
\widehat{x_t}\dots \wedge x_{p+1}),
\end{split}
\end{equation*}
for $p\ge 0$ and $\om \in C^p(\Lg,M)$. If $p<0$ then we set $d_p=0$.
A standard computation shows $d_p\circ d_{p-1}=0$, hence the definition
$$ H^p(\Lg,M)=\ker d_p /\im d_{p-1}=Z^p(\Lg,M)/B^p(\Lg,M) $$
makes sense. This space is called the {\it $p^{th}$ cohomology group}
of $\Lg$ with coefficients in the $\Lg$--module $M$. The elements from
$Z^p(\Lg,M)$ are called {\it $p$--cocycles}, and from $B^p(\Lg,M)$
{\it $p$--coboundaries}.
The sequence
\begin{equation*}
0 \rightarrow C^0(\Lg,M) \xrightarrow{d_0} C^1(\Lg,M) \xrightarrow{d_1}
C^2(\Lg,M) \rightarrow \cdots
\end{equation*}
yields a cochain complex, which is called the {\it standard} cochain complex
and is denoted by $\{ C^{\pkt}(\Lg,M),d \}$.\\
The space of $2$--cocycles and $2$--coboundaries is given explicitly as follows:

\begin{equation*}
\begin{split}
Z^2(\Lg,M) & =\{ \om \in \Hom(\Lambda^2 \Lg,M) \mid x_1\pkt \om
(x_2\wedge x_3)-
x_2\pkt \om (x_1\wedge x_3)+x_3\pkt \om (x_1\wedge x_2)\\
 & \hskip0.25 cm -\om ([x_1,x_2]\wedge x_3) + \om ([x_1,x_3]\wedge x_2) -
\om ([x_2,x_3]\wedge x_1)=0 \}\\
B^2(\Lg,M) & =\{ \om \in \Hom(\Lambda^2\Lg,M) \mid \om (x_1\wedge x_2)=x_1\pkt
f(x_2)-x_2\pkt f(x_1)-f([x_1,x_2])\\
 & \qquad \mbox{for some } f\in \Hom (\Lg,M)\}
\end{split}
\end{equation*}

There are important special cases of Lie algebra cohomology.
If $M=K$ denotes the $1$--dimensional trivial module, i.e., $x\pkt m=0$
for all $x\in\Lg$, then the numbers $b_p(\Lg)=\dim H^p(\Lg,K)$ are of special
interest. The number $b_p(\Lg)$ is called the {\it $p^{th}$ Betti number}.
There are many questions regarding the Betti numbers of nilpotent Lie algebras.
Among other things one would like to know good upper and lower bounds for each $b_p(\Lg)$.
It is still an open conjecture whether or not the following is true for
nilpotent Lie algebras:
$$b_2(\Lg) > \frac{b_1(\Lg)^2}{4}$$
This is called the {\it $b_2$--conjecture}. It is
proved for algebras with $b_1(\Lg)\le 3$, for $2$--step nilpotent Lie algebras
and for nilpotent Lie algebras $\Lg$ with
$ \dim \Lg /\Lz (\Lg) \le 7$. For details see \cite{CJP}.
Another conjecture is the {\it toral rank conjecture} for nilpotent
Lie algebras, stating
\begin{equation*}
\sum_{p=0}^n b_p(\Lg)\ge 2^{\dim \Lz(\Lg)}
\end{equation*}
where $\Lz(\Lg)$ denotes the center of $\Lg$. That is
also known only in few cases \cite{CJ}. For filiform
Lie algebras however both conjectures are clear. Nevertheless the
explicit determination of Betti numbers of filiform algebras 
leads to formidable combinatorial problems, see \cite{AS}. \\[0.3cm]
We come now to the definition of an affine $2$--cocycle:

\begin{defi}
Let $\Lg\in \LF_n(K)$. A $2$--cocycle $\om\in Z^2(\Lg,K)$
is called {\it affine}, if $\om : \Lg \wedge \Lg \rightarrow K$ is nonzero
on $\Lz (\Lg) \wedge \Lg$. A class $[\om]\in H^2(\Lg,K)$ is called
affine if every representative is affine.
\end{defi}
\begin{lem}
Let $\Lg\in \LF_n(K)$ and $\om\in Z^2(\Lg,K)$ be an affine $2$--cocycle.
Then its cohomology class $[\om]\in H^2(\Lg,K)$ is affine and nonzero.
\end{lem}
\begin{proof}
If $\Lz(\Lg)=\s \{z\}$, then $\om$ is affine iff $\om(z\wedge y)\ne 0$ for
some $y\in \Lg$. For $\xi \in B^2(\Lg,K)$ we have $\xi(z\wedge y)=f([z,y])=
f(0)=0$ for some linear form $f\in \Lg^*$. Hence $\om$ is not a
$2$--coboundary and $[\om]$ is affine.
\end{proof}

Since the elements of $H^2(\Lg,K)$ classify the equivalence
classes of central extensions of $\Lg$ by $K$ we obtain the following
characterization:

\begin{prop} \label{h2}
A Lie algebra $\Lg\in\LF_n(K)$ has an extension
\begin{equation} \label{zen}
0 \rightarrow \Lz (\Lh) \xrightarrow{\iota} \Lh \xrightarrow{\pi}
\Lg \rightarrow 0
\end{equation}
with $\Lh \in \LF_{n+1}(K)$ if and only if there exists an affine
$[\om]\in H^2(\Lg,K)$.
\end{prop}
\begin{proof}
The center $\Lz(\Lg)=\s \{z \}$ is one-dimensional.
Suppose that $\Lg$ has such an extension. Then $\Lz (\Lh)$ is a trivial
$\Lg$--module equal to $K$.
The extension determines a unique class $[\om] \in H^2(\Lg,K)$
and we may assume that the Lie bracket is given by
\begin{equation}
\label{h}
[(a,x),(b,y)]_{\Lh}:=(\om (x\wedge y),[x,y]_{\Lg})
\end{equation}
on the vector space $\Lh:=K\oplus \Lg$.
Suppose that $\om(z \wedge y)=0$ for all $y\in \Lg$. Then
$(a,0)$ and $(a,z)$ are contained in $\Lz (\Lh)$. This contradicts
$\Lz(\Lh)\cong K$. Hence $\om$ is affine.

Conversely an affine $[\om]\in H^2(\Lg,K)$ determines an extension
\begin{equation*}
0 \rightarrow K \xrightarrow{\iota} \Lh \xrightarrow{\pi}
\Lg \rightarrow 0
\end{equation*}
via the Lie bracket $(\ref{h})$
on $\Lh:=K\oplus \Lg$. Let $(a,x)\in \Lz (\Lh)$.
Then $x\in \Lz(\Lg)$ and it follows that $x$ is a multiple of $z$.
Since $\om (z,y) \ne 0$ for some $y\in \Lg$ it follows that
$(a,z)$ is not in $\Lz (\Lh)$. Hence $x=0$, $\Lz (\Lh)$ is the trivial
one-dimensional $\Lg$--module $K$ and $\Lh \in \LF_{n+1}(K)$.
\end{proof}

There is the following result on the connection between affine
cohomology classes and affine structures:

\begin{prop}
Let $\Lg\in\LF_n(K)$ and assume that there exists an affine cohomology class
$[\om]\in H^2(\Lg,K)$. Then $\Lg$ admits an affine structure.
\end{prop}

The proposition is a corollary of the following theorem:

\begin{thm} \label{ext}
Let $\Lg\in\LF_n(K)$ and suppose that $\Lg$
has an extension
\begin{equation*}
0 \rightarrow \La \xrightarrow{\iota} \Lh \xrightarrow{\pi}
\Lg \rightarrow 0
\end{equation*}
with $\iota(\La)=\Lz(\Lh)$. Then $\Lg$ admits an affine structure.
\end{thm}  

\begin{proof}
The first step of the proof consists in showing that
we may assume $\Lh\in \LF_{n+1}(K)$. For this we
refer the reader to the extended version of this article.
Let $(f_1,\dots ,f_{n+1})$ be an adapted basis for $\Lh$ and let
$e_i:=f_i$ mod $\Lz (\Lh)$ for $i=1,\dots, n$.
Then $(e_1,\dots ,e_n)$ is an adapted basis of $\Lg$. Let
$\Lh_3=\s \{f_3,\ldots ,f_{n+1}\}$ and $\Lg_2=\s \{e_2,\ldots ,e_{n}\}$.
There is a uniquely determined linear map $\phi : \Lg \ra \Lh_3$
satisfying $\phi (x)=[f_1,\ov{x}]_{\Lh}$ for all $x\in \Lg$ where
$\ov{x}\in \Lh$ is any element with $\pi (\ov{x})=x$. The restriction
of $\phi$ to $\Lg_2$ is bijective since it is evidently injective. Denote
its inverse by $\psi : \Lh_3 \ra \Lg_2$. Now set for all $x,y\in \Lg$    
\begin{equation}
\label{lsa}
x\pkt y:=\psi ([\ov{x},\phi (y)]_{\Lh})
\end{equation}
The formula is well defined since $[\ov{x},\phi (y)]_{\Lh}
=[\ov{x},[f_1,\ov{y}]]\in \Lh_3$.
We will show that it satisfies conditions \eqref{lsa1} and \eqref{lsa2} of
Definition $\ref{affine}$ and hence defines an affine structure on $\Lg$:

\begin{align*}
x\pkt y-y\pkt x & = \psi([\ov{x},[f_1,\ov{y}]]-[\ov{y},[f_1,\ov{x}]]) \\
 & = \psi([\ov{y},[f_1,\ov{x}]]-[f_1,[\ov{y},\ov{x}]] -[\ov{y},
[f_1,\ov{x}]]) \\
 & = \psi([f_1,[\ov{x},\ov{y}]])=\psi([f_1,\ov{[x,y]}_{\Lg}])
     =\psi(\phi([x,y]_{\Lg})) \\
 & = [x,y]_{\Lg}
\end{align*}
where the brackets are taken in $\Lh$ if not otherwise denoted.
Using the identity $[f_1,\ov{\psi(w)}]=w$ for all $w\in \Lh_3$
and again the Jacobi identity we obtain for all $x,y,z\in \Lg$:

\begin{align*}
x\pkt (y\pkt z)-y\pkt (x\pkt z) & = \psi ([\ov{x},[\ov{y},\phi(z)]]-
[\ov{y},[\ov{x},\phi(z)]]) \\
 & = [x,y]_{\Lg}\pkt z \\
 & = (x\pkt y) \pkt z-(y\pkt x)\pkt z
\end{align*}

\end{proof}    

\section{Computation of $H^2(\Lg,K)$}

In this section we will completely determine the
cohomology groups $H^2(\Lg,K)$ for all $\Lg\in \LF_n(K)$ with $n\le 11$.
We will also give some results for algebras from the classes
$\LA_n^1(K)$ and $\LA_n^2(K)$.
The cohomology spaces give important information on $\Lg$. 
In our case, we obtain a complete description of the
existence of affine cohomology classes. 
Let $(e_1,\ldots ,e_n)$ be an adapted basis for $\Lg$
so that its Lie algebra law lies in $\CA_n(K)$.

\begin{lem}
Let $\om\in \Hom (\Lambda^2\Lg,K)$. Then $\om$ is an affine $2$--cocycle
iff $\om (e_1\wedge e_n)$ or $\om (e_2\wedge e_n)$ is nonzero.
\end{lem}
\begin{proof}
By definition, $\om$ is affine iff $\om (e_j\wedge e_n)\ne 0$ for some
$j\in \{1,\ldots ,n\}$. The condition for $\om$ to be a $2$--cocycle
is as follows:
\begin{equation}\label{ijk}
\om([e_i,e_j]\wedge e_k)-\om([e_i,e_k]\wedge e_j)+\om([e_j,e_k]\wedge
e_i)=0 \quad \text{ for } i<j<k
\end{equation}
Setting $i=1,k=n$ we obtain $\om(e_j \wedge e_n)=0$ for $3\le j\le n$.
\end{proof}

\begin{defi}
Define $\om_{\ell}\in \Hom (\Lambda^2\Lg,K)$ by the nonzero values as follows:
\begin{align}\label{omega}
\om_{\ell}(e_k\wedge e_{2\ell+3-k}) & = (-1)^k \quad \text{ for } 1\le \ell
\le [(n-1)/2],\; 2\le k\le \left[(2\ell+3)/2\right]
\end{align}
\end{defi}
In the following we will mainly use $\om_1,\dots,\om_4$.
They are defined by
\begin{align*}
\om_1(e_2\wedge e_3)& = 1 \\
\om_2(e_2\wedge e_5)& = 1,\;\om_2(e_3\wedge e_4) = -1\\
\om_3(e_2\wedge e_7)& = 1,\;\om_3(e_3\wedge e_6) = -1,\;
\om_3(e_4\wedge e_5) = 1\\
\om_4(e_2\wedge e_9)& = 1,\;\om_4(e_3\wedge e_8) = -1,\;
\om_4(e_4\wedge e_7) = 1,\;\om_4(e_5\wedge e_6) = -1
\end{align*}
\begin{lem}
We have $\om_1,\om_2\in Z^2(\Lg,K)$, whereas $\om_{\ell},\,\ell \ge 3$
need not be $2$--cocycles. If $\ell <[(n-1)/2]$, then $\om_{\ell}$
cannot be an affine $2$--cocycle.
\end{lem}
\begin{proof}
The first claim follows easily from equation \eqref{ijk}.
In the case $i=1,j=2$ it reduces to $\om_1(e_3\wedge e_k)=
\om_1(e_{k+1}\wedge e_2)$ for $k\ge 3$.
On the other hand we have $\om_{\ell}(e_i\wedge e_n)=0,\, 1\le i\le n$ for
$\ell <[(n-1)/2]$.
\end{proof}

It is interesting to consider
filiform Lie algebras with minimal
second Betti number. It is not difficult to show the
following result:

\begin{prop} \label{be2}
Let $\Lg\in\LF_n(K),\, n\ge 6$ be a filiform Lie algebra with $b_2(\Lg)=2$.
Then there exists no affine $[\om]\in H^2(\Lg,K)$.
\end{prop}

We come now to the computation of the cohomology.  
We have to divide the filiform algebras into several well defined classes.
However, the number of classes should be as small as possible.
Hence we do not use the classification results of filiform Lie algebras. 
We divide $\CA_n(K),\,6\le n\le 11$
into the following subsets depending on certain equalities or inequalities
of the structure constants. These subsets correspond to well defined
classes of filiform Lie algebras:

\vspace*{0.5cm}
\begin{tabular}{|c|c|}
\hline
Class & Conditions \\
\hline\hline
$\CA_{6,1}$ & $\al_{3,6}\ne 0$ \\ \hline
$\CA_{6,2}$ & $\al_{3,6}= 0$   \\ \hline
$\CA_{7,1}$ & $2\al_{2,5}+\al_{3,7}\ne 0$   \\ \hline
$\CA_{7,2}$ & $2\al_{2,5}+\al_{3,7}=   0$   \\ \hline
$\CA_{8,1}$ & $\al_{4,8}\ne 0,\,2\al_{2,5}+\al_{3,7}=0$   \\ \hline
$\CA_{8,2}$ & $\al_{4,8}=0,\,2\al_{2,5}+\al_{3,7}\ne 0$   \\ \hline
$\CA_{8,3}$ & $\al_{4,8}=0,\,2\al_{2,5}+\al_{3,7}=0,\, \al_{2,5}\ne 0$   \\ \hline
$\CA_{8,4}$ & $\al_{2,5}=\al_{3,7}=\al_{4,8}=0$   \\ \hline
$\CA_{9,1}$ & $2\al_{2,5}+\al_{3,7}\ne 0,\, \al_{3,7}^2\ne \al_{2,5}^2$ \\ \hline
$\CA_{9,2}$ & $2\al_{2,5}+\al_{3,7}\ne 0,\, \al_{3,7}^2=   \al_{2,5}^2$ \\ \hline
$\CA_{9,3}$ & $\al_{2,5}=\al_{3,7}= 0,\, \al_{4,9}\ne 0,\, \al_{2,6}+\al_{3,8}\ne 0$ \\ \hline
$\CA_{9,4}$ & $\al_{2,5}=\al_{3,7}= 0,\, \al_{4,9}\ne 0,\, \al_{2,6}+\al_{3,8}= 0$ \\ \hline
$\CA_{9,5}$ & $\al_{2,5}=\al_{3,7}=\al_{4,9}=0,\, 2\al_{2,7}+\al_{3,9}\ne 0$ \\ \hline
$\CA_{9,6}$ & $\al_{2,5}=\al_{3,7}=\al_{4,9}=0,\, 2\al_{2,7}+\al_{3,9}= 0$ \\ \hline 
$\CA_{10,1}$ & $\al_{5,10}\ne 0,\,2\al_{2,5}+\al_{3,7}\ne 0$ \\ \hline
$\CA_{10,2}$ & $\al_{5,10}\ne 0,\,2\al_{2,5}+\al_{3,7}=   0$ \\ \hline                                                  
$\CA_{10,3}$ & $\al_{5,10}= 0,\,2\al_{2,5}+\al_{3,7}\ne 0,\, \al_{3,7}^2\ne \al_{2,5}^2 $ \\ \hline              
$\CA_{10,4}$ & $\al_{5,10}= 0,\,2\al_{2,5}+\al_{3,7}\ne 0\, \al_{3,7}^2= \al_{2,5}^2$ \\ \hline                            
$\CA_{10,5}$ & $\al_{5,10}= 0,\,2\al_{2,5}+\al_{3,7}= 0,\, \al_{4,9}\ne 0,\, \al_{2,6}^2+2\al_{2,7}
\al_{4,9}\ne 0$ \\ \hline                          
$\CA_{10,6}$ & $\al_{5,10}= 0,\,2\al_{2,5}+\al_{3,7}= 0,\, \al_{4,9}\ne 0,\,\al_{2,6}^2+2\al_{2,7}
\al_{4,9}= 0$ \\ \hline                                    
$\CA_{10,7}$ & $\al_{5,10}= 0,\,2\al_{2,5}+\al_{3,7}= 0,\,\al_{4,9}=0,\, 2\al_{2,7}+\al_{3,9}\ne 0$ \\ \hline             
$\CA_{10,8}$ & $\al_{5,10}= 0,\,2\al_{2,5}+\al_{3,7}= 0,\,\al_{4,9}=0,\, 2\al_{2,7}+\al_{3,9}= 0
,\, \al\ne 0$ \\ \hline                                       
$\CA_{10,9}$ & $\al_{5,10}= 0,\,2\al_{2,5}+\al_{3,7}= 0,\,\al_{4,9}=0,\,2\al_{2,7}+\al_{3,9}= 0 
,\, \al= 0$ \\ \hline                                          
%
$\CA_{11,1}$ & $2\al_{2,5}+\al_{3,7}\ne 0,\, 10\al_{3,7}-\al_{2,5}\ne 0,\, \be \ne 0$ \\ \hline
$\CA_{11,2}$ & $2\al_{2,5}+\al_{3,7}\ne 0,\, 10\al_{3,7}-\al_{2,5}\ne 0,\, \be = 0$ \\ \hline
$\CA_{11,3}$ & $2\al_{2,5}+\al_{3,7}\ne 0,\,10\al_{3,7}-\al_{2,5}= 0 $ \\ \hline
$\CA_{11,4}$ & $2\al_{2,5}+\al_{3,7}= 0,\, \al_{4,9}\ne 0$ \\ \hline
$\CA_{11,5}$ & $\al_{2,5}=\al_{3,7}=\al_{4,9}=0,\, \al_{5,11}\ne 0,\, 4\al_{4,10}+2\al_{3,8}-3\al_{2,6}\ne 0,
\al \ne 0$ \\ \hline
$\CA_{11,6}$ & $\al_{2,5}=\al_{3,7}=\al_{4,9}=0,\, \al_{5,11}\ne 0,\, 4\al_{4,10}+2\al_{3,8}-3\al_{2,6}\ne 0,\,
\al = 0$ \\ \hline 
$\CA_{11,7}$ & $\al_{2,5}=\al_{3,7}=\al_{4,9}=0,\, \al_{5,11}\ne 0,\,
4\al_{4,10}+2\al_{3,8}-3\al_{2,6}= 0,\, \ga \ne 0 $ \\ \hline
$\CA_{11,8}$ & $\al_{2,5}=\al_{3,7}=\al_{4,9}=0,\, \al_{5,11}\ne 0,\, 4\al_{4,10}+2\al_{3,8}-3\al_{2,6}
= 0,\, \ga = 0,\, \de \ne 0$ \\ \hline
$\CA_{11,9}$ & $\al_{2,5}=\al_{3,7}=\al_{4,9}=0,\, \al_{5,11}\ne 0,\, 4\al_{4,10}+2\al_{3,8}-3\al_{2,6}
= 0,\, \ga = 0,\, \de = 0$ \\ \hline
$\CA_{11,10}$ & $\al_{2,5}=\al_{3,7}=\al_{4,9}=\al_{5,11}=0$ \\ \hline 
\end{tabular}     

\vspace*{0.5cm}
where

\begin{align*}
\al & = 3\al_{4,10}(\al_{2,6}+\al_{3,8})-4\al_{3,8}^2 \\
\be & =(2\al_{2,5}^2-5\al_{3,7}^2)(4\al_{2,5}^2-4\al_{2,5}\al_{3,7}+3\al_{3,7}^2)\\
\ga & = 22\al_{3,8}^2-3\al_{2,6}\al_{3,8}-9\al_{2,6}^2 \\
\de & = \al_{5,11}(4\al_{3,10}+5\al_{2,8})-3\al_{4,11}(\al_{2,6}+\al_{3,8})+
2\al_{2,7}(3\al_{2,6}-11\al_{3,8})
\end{align*}

The result of the computation is as follows:\\

\begin{prop}
The following table shows the cohomology
spaces $H^2(\Lg_{\la},K)$ for all $\la\in \CA_n(K),\, 3\le n\le 11$. 
The corresponding Lie algebras $\Lg_{\la}$ 
admit an affine cohomology class as follows:\\[0.5cm]
\hspace*{1.8cm}
\begin{tabular}{|c|c|c|c|c|}
\hline
$\dim \Lg_{\la}$ & Class & $H^2(\Lg_{\la},K)$  & affine $\om$ & $b_2(\Lg_{\la})$ \\
\hline\hline
$3$ & $\CA_3$     & $\om_1,\om$       & $\checkmark$   & $2$ \\ \hline 
$4$ & $\CA_4$     & $\om_1,\om$       & $\checkmark$   & $2$ \\ \hline
$5$ & $\CA_5$     & $\om_1,\om_2,\om$ & $\checkmark$   & $3$ \\ \hline
$6$ & $\CA_{6,1}$ & $\om_1,\om_2$     & $\mi$          & $2$ \\ \hline
$6$ & $\CA_{6,2}$ & $\om_1,\om_2,\om$ & $\checkmark$   & $3$ \\ \hline
$7$ & $\CA_{7,1}$ & $\om_1,\om_2,\om$ & $\checkmark$   & $3$ \\ \hline
$7$ & $\CA_{7,2}$ & $\om_1,\om_2,\om_3,\om$ & $\checkmark$  & $4$ \\ \hline
$8$ & $\CA_{8,1}$ & $\om_1,\om_2,\om_3$     & $\mi$         & $3$ \\ \hline
$8$ & $\CA_{8,2}$ & $\om_1,\om_2,\om$ & $\checkmark$   & $3$ \\ \hline
$8$ & $\CA_{8,3}$  & $\om_1,\om_2,\om_3$ & $\mi$          & $3$ \\ \hline
$8$ & $\CA_{8,4}$  & $\om_1,\om_2,\om_3,\om$ & $\checkmark$ & $4$ \\ \hline
$9$ & $\CA_{9,1}$  & $\om_1,\om_2,\om$ & $\checkmark$   & $3$ \\ \hline
$9$ & $\CA_{9,2}$  & $\om_1,\om_2,\om,\om'$ & $\checkmark$   & $4$ \\ \hline
$9$ & $\CA_{9,3}$  & $\om_1,\om_2,\om_3$ & $\mi$          & $3$ \\ \hline
$9$ & $\CA_{9,4}$  & $\om_1,\om_2,\om_3,\om$ & $\checkmark$   & $4$ \\ \hline
$9$ & $\CA_{9,5}$  & $\om_1,\om_2,\om_3,\om$ & $\checkmark$   & $4$ \\ \hline
$9$ & $\CA_{9,6}$  & $\om_1,\om_2,\om_3,\om,\om'$ &$\checkmark$& $5$ \\ \hline
$10$& $\CA_{10,1}$ & $\om_1,\om_2,\om_3$ & $\mi$         & $3$ \\ \hline
$10$& $\CA_{10,2}$ & $\om_1,\om_2,\om_3,\om_4$ & $\mi$       & $4$ \\ \hline
$10$& $\CA_{10,3}$ & $\om_1,\om_2,\om$   & $\checkmark$  & $3$ \\ \hline
$10$& $\CA_{10,4}$ & $\om_1,\om_2,\om_3$  & $\mi$         & $3$ \\ \hline
$10$& $\CA_{10,5}$ & $\om_1,\om_2,\om_3$  & $\mi$         & $3$ \\ \hline
$10$& $\CA_{10,6}$ & $\om_1,\om_2,\om_3,\om$  & $\checkmark$  & $4$ \\ \hline
$10$& $\CA_{10,7}$ & $\om_1,\om_2,\om_3,\om$  & $\checkmark$  & $4$ \\ \hline
$10$& $\CA_{10,8}$ & $\om_1,\om_2,\om_3,\om_4$ & $\mi$        & $4$ \\ \hline
$10$& $\CA_{10,9}$ & $\om_1,\om_2,\om_3,\om_4,\om$&$\checkmark$& $5$ \\ \hline
%
$11$& $\CA_{11,1}$ & $\om_1,\om_2$  & $\mi$         & $2$ \\ \hline
$11$& $\CA_{11,2}$ & $\om_1,\om_2,\om$   & $\checkmark$  & $3$ \\ \hline
$11$& $\CA_{11,3}$ & $\om_1,\om_2,\om$   & $\checkmark$  & $3$ \\ \hline
$11$& $\CA_{11,4}$ & $\om_1,\om_2,\om_3$ & $\mi$         & $3$ \\ \hline
$11$& $\CA_{11,5}$ & $\om_1,\om_2,\om_3$ & $\mi$         & $3$ \\ \hline 
$11$& $\CA_{11,6}$ & $\om_1,\om_2,\om_3,\om_4$ & $\mi$   & $4$ \\ \hline 
$11$& $\CA_{11,7}$ & $\om_1,\om_2,\om_3,\om$ & $\checkmark$  & $4$ \\ \hline
$11$& $\CA_{11,8}$ & $\om_1,\om_2,\om_3,\om_4$ & $\mi$   & $4$ \\ \hline 
$11$& $\CA_{11,9}$ & $\om_1,\om_2,\om_3,\om_4,\om$ & $\checkmark$   & $5$ \\ \hline 
\end{tabular} 
\end{prop}
\vspace{0.3cm}
The notations here are as follows. By $\om_1,\ldots,\om_4$ we denote always
the $2$--cocycles defined by \eqref{omega}. The letter $\om$ stands for an 
affine $2$--cocycle,
which might be different for distinct classes of Lie algebras.
The same holds for $\om'$. For the cohomology spaces the table shows representing
$2$--cocycles for a basis. So $\om_1,\om_2,\om$ in the table means that
$([\om_1],[\om_2],[\om])$ is a basis of $H^2(\Lg_{\la},K)$.
Note that for $n\ge 5$ the two-dimensional subspace spanned by
$[\om_1],[\om_2]$ is always contained in $H^2(\Lg_{\la},K)$.
A checkmark denotes the existence and a minus sign the absence of an
affine $2$--cocycle. 

\begin{rem}
We have also determined the
cohomology and the affine cohomology classes for the
class $\CA_{11,10}$. 
However, to state the result here would us require to introduce too 
many subclasses. 
On the other hand, it is not difficult to show that
all such algebras admit an affine structure. 
\end{rem}

\begin{rem}
Let $\la\in \CA_{6,1}$. Then $\Lg_{\la}$ does not admit an affine
$[\om]\in H^2(\Lg_{\la},K)$. However, $\Lg_{\la}$ admits an affine
structure since there exists a nonsingular derivation.
\end{rem} 

The computations have been done with the computer algebra
package REDUCE. For the Lie algebras of the classes 
$\LA_n^1(K)$ and $\LA_n^2(K)$ we have obtained the following results.

\begin{thm}
Let $\Lg \in \LA_n^1(K),\, n\ge 12$. Then $b_2(\Lg)\ge 3$ and
there always exists an affine cohomology class. Hence all algebras
$\Lg \in \LA_n^1(K),\, n\ge 12$ admit a canonical affine structure.
\end{thm}

\begin{prop}
Let $\Lg\in \LA^2_{12}(K)$. Then $H^2(\Lg,K)=\s \{[\om_1],[\om_2],[\om]\}$,
where $\om$ is an affine $2$--cocycle.
\end{prop}

The last result generalizes to higher dimensions as follows: 
for $\Lg\in \LA^2_{n}(K),\, n\ge 13$ the existence of an affine
cohomology class depends on one certain polynomial condition $P_n\equiv 0$
in the structure constants of $\Lg$:

\begin{thm}
Let $\Lg\in \LA^2_{n}(K),\, n\ge 13$. Then
\begin{equation*}
H^2(\Lg,K)=\begin{cases}
\s\{[\om_1],[\om_2],[\om]\} & \text{if $\Lg$ satisfies $P_n\equiv 0$ },\\[0.2cm]
\s\{[\om_1],[\om_2]\} & \text{otherwise}.
\end{cases}
\end{equation*} 
\end{thm}

In particular, $b_2(\Lg)=2$ for all algebras $\Lg\in \LA^2_n(K),\,n\ge 13$
not satisfying the polynomial condition $P_n\equiv 0$. 
These algebras have minimal second Betti numbers (among nilpotent Lie algebras)
and are candidates for Lie algebras without affine structures. In that direction
we could prove:

\begin{thm}
No Lie algebra $\Lg\in \LA^2_{13}(K)$ admits any affine structure.
\end{thm}

More precisely we have proved that $\mu (\Lg)\ge 15$, where $\mu (\Lg)$
denotes the minimal dimension of a faithful $\Lg$--module.
If $\Lg$ admits an affine structure then it is not difficult to see that
$\mu (\Lg)\le \dim \Lg +1$. Hence the existence question of affine structures
is connected to the question of a refinement of Ado's theorem.
It is very difficult to compute the invariant $\mu(\Lg)$, in particular
for nilpotent Lie algebras. Then the adjoint representation is not
faithful. 
For abelian Lie algebras the invariant can be computed explicitly:
denote by $\lceil x \rceil$ the ceiling of $x$, i.e., the least integer
greater than or equal to $x$ and let
$\Lg$ be an abelian Lie algebra of dimension $n$ over an arbitrary
field $K$. Then $\mu(\Lg)= \lceil 2\sqrt{n-1}\rceil $.     
The nilpotent Lie algebras which are counterexamples to
the Milnor conjecture satisfy $\mu (\Lg)\ge \dim \Lg +2$.
For details see \cite{BU5}.
Returning to the cohomology $H^2(\Lg,K)$ we think it is interesting 
to study the following problem:

\begin{prob*}
Does a Lie algebra $\Lg\in \LA^2_{n}(K),\, n\ge 13$ satisfy
$\mu(\Lg)\ge n+2$ if and only if there is no affine $[\om]\in H^2(\Lg,K)$ ?
\end{prob*}

\end{document}